\documentclass[11pt,reqno]{amsart}
\usepackage[margin=1in]{geometry}
\usepackage[T1]{fontenc}
\usepackage{lmodern}
\usepackage{amsmath,amssymb,amsthm,mathtools}
\usepackage[hidelinks]{hyperref}

\numberwithin{equation}{section}

\newtheorem{theorem}{Theorem}[section]
\newtheorem{proposition}[theorem]{Proposition}
\newtheorem{lemma}[theorem]{Lemma}
\newtheorem{corollary}[theorem]{Corollary}
\theoremstyle{remark}
\newtheorem{remark}[theorem]{Remark}

\newcommand{\F}{\mathbb F}
\newcommand{\Sn}{\mathfrak S_n}
\newcommand{\Hom}{\operatorname{Hom}}

\newcommand{\JSp}{\operatorname{Sp}}
\newcommand{\RS}{\mathsf S}
\newcommand{\CS}{\mathsf S}

\title[Specht Modules With Trivial Direct Summands]{Specht Modules With Trivial Direct Summands}
\author{David J. Hemmer}
\address{Department of Mathematical Sciences\\
  Michigan Technological University\\
  Houghton, MI 49931}
\email{djhemmer@mtu.edu}
\date{August 13, 2026}

\begin{document}

\begin{abstract}
We prove that, in characteristic $2$, a Specht module indexed by a nonhook partition cannot have a trivial direct summand.  The proof uses the KLR grading and a recent graded homomorphism theorem of Hudak.  Together with Murphy's theorem for hooks, this gives the complete classification of trivial direct summands of Specht modules in characteristic $2$, answering a question posed by Collins--Dodge and Dodge--Van Vlack.
\end{abstract}

\maketitle

\section{Introduction}

Let $\lambda\vdash n$, and let $\JSp^\lambda$ denote the Specht module for the symmetric group $\Sn$ as defined in James \cite{James}.  Over a field $\F$ of characteristic $2$, the trivial and sign representations of $\Sn$ are isomorphic, so there is a unique one-dimensional $\F\Sn$-module up to isomorphism.

Murphy \cite{Murphy} classified the hook Specht modules having a one-dimensional direct summand.  Note that in odd characteristic, $Sp^\lambda$ is always indecomposable \cite{James}, so this problem is only interesting in characteristic $2$. 

The decomposition of hook Specht modules in characteristic $2$ was subsequently determined more explicitly by Donkin and Geranios \cite{DonkinGeranios}, who described these modules as direct sums of indecomposable Young modules.  Thus the structure of the hook case is considerably better understood than that of general Specht modules.  Our concern here is the more specific question of when the trivial module can occur as a direct summand.  Theorem~\ref{thm:main} shows that this phenomenon is confined entirely to the hook case.

Collins and Dodge \cite{CollinsDodge} raised the general question of determining when a Specht module has a trivial direct summand.  The spaces
\[
\Hom_{\Sn}(\F,\JSp^\lambda)
\qquad\text{and}\qquad
\Hom_{\Sn}(\JSp^\lambda,\F)
\]
are both at most one-dimensional, and their nonvanishing is completely understood.  Thus the question reduces to determining whether the composition
\[
\F \xrightarrow{i_\lambda} \JSp^\lambda
\xrightarrow{p_\lambda} \F
\]
is nonzero.

Collins and Dodge \cite[Theorem~3.5]{CollinsDodge} gave a combinatorial parity condition for this to happen.  Dodge and Van Vlack \cite{DodgeVanVlack} used it to prove that the Specht modules indexed by partitions $(k,3,3)$ have no one-dimensional summand (despite both maps $i_\lambda$ and $p_\lambda$ being nonzero).  The purpose of this note is to answer the Collins--Dodge question negatively: no nonhook partition can occur.

\begin{theorem}\label{thm:main}
Let $\F$ be a field of characteristic $2$, and let $\lambda\vdash n$ be a nonhook partition.  Then $\JSp^\lambda$ has no one-dimensional direct summand.
\end{theorem}

The proof uses two graded representation-theoretic inputs.  Hudak's Theorem~4.38 \cite{Hudak} determines the graded Hom-space from a graded row Specht module to the one-dimensional trivial module.  We apply Hudak's theorem twice: the application to $\lambda'$ gives the degree of the quotient map from the graded lift of James's Specht module, while the application to $\lambda$, after graded duality, gives the degree of the inclusion map.  The graded row/column duality and conjugation results of Kleshchev--Mathas--Ram \cite{KMR} provide the bridge between Hudak's row Specht modules and the classical James convention, including the required grading shift.

Combining Theorem~\ref{thm:main} with Murphy's hook theorem yields the full classification of Specht modules with trivial direct summands; see Corollary~\ref{cor:classification}. 

\section{Graded Specht modules and the two Hom-spaces}

Let $\RS^\mu$ denote the universal graded \emph{row} Specht module of Kleshchev--Mathas--Ram \cite{KMR}, and let $\CS_\mu$ denote their graded \emph{column} Specht module.  After forgetting the grading, $\CS_\mu$ is the Specht module $\JSp^\mu$ in the convention of James \cite[p.~1246]{KMR}.

Throughout this section, $\Hom_A(M,N)$ denotes the Hom-space for the underlying ungraded $A$-modules.  If $M$ and $N$ are graded, write $\Hom_A^d(M,N)$ for the subspace of homogeneous homomorphisms of degree $d$.  Since all modules occurring here are finite dimensional,
\[
\Hom_A(M,N)=\bigoplus_{d\in\mathbb Z}\Hom_A^d(M,N).
\]
Thus a statement that the graded Hom-space has graded dimension $q^d$ means that the ordinary Hom-space is one-dimensional and is spanned by a homogeneous map of degree $d$.

We work at quantum characteristic $e=2$, level one with multicharge $0$, and $\operatorname{char}\F=2$.  Via the Brundan--Kleshchev identification, the corresponding cyclotomic KLR algebra is identified, as an ungraded algebra, with the group algebra $\F\Sn$; throughout we use the KLR grading transported across this identification (see \cite{KMR}).  Put
\[
L:=\RS^{(1^n)}.
\]
As an ungraded module, $L$ is the unique one-dimensional $\F\Sn$-module.  In the grading used by Hudak \cite[Section~2.2]{Hudak}, it is concentrated in degree $0$.  We deliberately use the sign-labelled graded lift $\RS^{(1^n)}$, rather than $\RS^{(n)}$, because the former is concentrated in degree $0$ in Hudak's grading.

For a partition $\mu=(\mu_1,\ldots,\mu_z)$ set
\[
a(\mu):=\sum_{i=1}^z\left\lfloor\frac{\mu_i}{2}\right\rfloor.
\]
Hudak's result gives a complete set of necessary and sufficient conditions for the nonvanishing of $\Hom(\RS^\mu,L)$, together with its graded dimension; we record below only the parts that we will need:

\begin{proposition}\cite[Theorem~4.38]{Hudak}\label{prop:hudak}
Let $\mu=(\mu_1,\ldots,\mu_z)\vdash n$ and suppose $\operatorname{char}\F=2$.  Then
\[
\Hom(\RS^\mu,L)
\]
is either zero or one-dimensional.  If it is nonzero, then
\[
\Hom^d(\RS^\mu,L)=0\qquad\text{for }d\ne-a(\mu),
\]
and $\Hom^{-a(\mu)}(\RS^\mu,L)$ is one-dimensional.  Moreover:
\begin{align*}
&\mu_i\text{ is odd for every }i<z,\\
&n\text{ odd }\Longrightarrow \mu_z\text{ is odd}.
\end{align*}
\end{proposition}

\begin{remark}\label{rem:hudak-extra}
Hudak's Theorem~4.38 imposes a further congruence condition for nonvanishing, which we do not need in what follows.
\end{remark}
Define
\[
D(\lambda):=\operatorname{def}(\operatorname{cont}(\lambda)),
\]
where $\operatorname{cont}(\lambda)$ and $\operatorname{def}$ are as in
\cite[(2.5), (2.7)]{KMR}. We use the grading-shift convention of \cite{KMR}:
\[
(M\langle d\rangle)_r=M_{r-d}.
\]
Thus $M\langle d\rangle$ is obtained by shifting the grading on $M$
up by $d$.

James proved \cite[Theorem 8.15]{James} that:
$$Sp^\lambda \otimes \operatorname{sgn} \cong (Sp^{\lambda'})^*.$$ Kleschev--Mathas--Ram gave a graded version that we will need: 

\begin{lemma}[Kleshchev--Mathas--Ram]\label{lem:duality}
For $e=2$ and $\operatorname{char}\F=2$ there are graded isomorphisms
\begin{align}
\CS_\lambda
&\cong \RS^{\lambda'},
\label{eq:column-row}\\
(\RS^\lambda)^{\circledast}
&\cong \CS_\lambda\langle-D(\lambda)\rangle.
\label{eq:row-dual-column}
\end{align}
\end{lemma}

\begin{proof}
The first isomorphism follows from
\cite[Theorem~8.5]{KMR}.  In our level-one setting the multicharge is $0$, and its conjugate multicharge is again $0$.  Thus Kleshchev--Mathas--Ram give
\[
\CS_\lambda\cong(\RS^{\lambda'})^{\mathrm{sgn}}.
\]
The sign automorphism is defined in \cite[(3.14)]{KMR}.  When $e=2$,
the residue involution $i\mapsto-i$ is trivial, and in characteristic
$2$ the minus signs on the generators $y_r$ and $\psi_s$ disappear.
Hence the sign twist is trivial in our setting.

For the second isomorphism,
\cite[Theorem~7.25]{KMR} gives
\[
\CS_\lambda
\cong
(\RS^\lambda)^{\circledast}\langle D(\lambda)\rangle,
\]
which is equivalent to
\[
(\RS^\lambda)^{\circledast}
\cong
\CS_\lambda\langle-D(\lambda)\rangle.
\]
\end{proof}
We can now identify separately the relevant inclusion and quotient module homomorphisms.  The quotient map comes directly from Hudak's theorem applied to $\lambda'$, while the inclusion map comes from dualizing Hudak's theorem applied to $\lambda$.
\begin{proposition}\label{prop:two-degrees}
Suppose both Hom-spaces of underlying ungraded modules are nonzero:
\begin{equation}\label{eq:two-hom-spaces}
\Hom(\CS_\lambda,L)\ne0,
\qquad
\Hom(L,\CS_\lambda)\ne0.
\end{equation}
Then each is one-dimensional.  They are spanned by homogeneous maps
\[
p_\lambda:\CS_\lambda\longrightarrow L,
\qquad
i_\lambda:L\longrightarrow\CS_\lambda
\]
with
\begin{equation}\label{eq:two-map-degrees}
\deg p_\lambda=-a(\lambda'),
\qquad
\deg i_\lambda=D(\lambda)-a(\lambda).
\end{equation}
Furthermore, the nonvanishing in \eqref{eq:two-hom-spaces}
implies the parity conditions of Proposition~\ref{prop:hudak}
for both $\lambda$ and $\lambda'$.
\end{proposition}

\begin{proof}
By \eqref{eq:column-row},
\[
\CS_\lambda\cong\RS^{\lambda'}.
\]
Hence
\[
\Hom(\CS_\lambda,L)
\cong
\Hom(\RS^{\lambda'},L).
\]
Proposition~\ref{prop:hudak}, applied to $\lambda'$, shows that this
ordinary Hom-space is one-dimensional when nonzero and is spanned by
a homogeneous map of degree $-a(\lambda')$.  It also gives the required
parity conditions for $\lambda'$.

For the inclusion, graded duality gives, after forgetting grading
shifts,
\[
\Hom(L,\CS_\lambda)
\cong
\Hom(\CS_\lambda^{\circledast},L^{\circledast}).
\]
Since $L$ is concentrated in degree $0$, we have
$L^{\circledast}\cong L$.  Dualizing \eqref{eq:row-dual-column} gives
\[
\CS_\lambda^{\circledast}\cong \RS^\lambda\langle-D(\lambda)\rangle,
\]
so, after forgetting the grading shift, nonvanishing of
$\Hom(L,\CS_\lambda)$ is equivalent to nonvanishing of
\[
\Hom(\RS^\lambda,L).
\]
Proposition~\ref{prop:hudak}, applied to $\lambda$, therefore shows
that $\Hom(L,\CS_\lambda)$ is one-dimensional and gives the required
parity conditions for $\lambda$.

To determine its degree explicitly, let
\[
q_\lambda:\RS^\lambda\longrightarrow L
\]
be a nonzero homomorphism.  By Proposition~\ref{prop:hudak},
$q_\lambda$ is homogeneous of degree $-a(\lambda)$.  Dualizing gives
a homogeneous map of the same degree
\[
L\longrightarrow(\RS^\lambda)^{\circledast}.
\]
Using \eqref{eq:row-dual-column}, its target may be identified with
\[
\CS_\lambda\langle-D(\lambda)\rangle.
\]
Thus this map has degree $-a(\lambda)$ with target
$\CS_\lambda\langle-D(\lambda)\rangle$.  Recall our grading-shift
convention $(M\langle d\rangle)_r=M_{r-d}$: a homogeneous map of
degree $d_0$ into $\CS_\lambda\langle-D(\lambda)\rangle$ is, as a map
into the unshifted module $\CS_\lambda$, homogeneous of degree
$d_0+D(\lambda)$.  With $d_0=-a(\lambda)$ this gives degree
\[
D(\lambda)-a(\lambda).
\]
This is $i_\lambda$, and proves the second formula in
\eqref{eq:two-map-degrees}.
\end{proof}

\section{The parity and defect calculation}

Assume from now on that $\lambda=(\lambda_1,\ldots,\lambda_h)\vdash n$ is a nonhook and that both Hom-spaces in \eqref{eq:two-hom-spaces} are nonzero.  Put $w=\lambda_1.$
Because $\lambda$ is a nonhook, it certainly has multiple rows and columns, so $h,w\ge2$.

\begin{lemma}\label{lem:parity-pattern}
Under these assumptions, $h$ and $w$ are odd.  When $n$ is odd, all
row lengths and column heights of $\lambda$ are odd.  When $n$ is even,
all but the last row and column are odd, while the last row length and
last column height are even.
\end{lemma}

\begin{proof}
By Proposition~\ref{prop:two-degrees}, Hudak's necessary parity conditions hold for $\lambda$.  Hence every row except possibly the last has odd length.  Since there are multiple rows, $w=\lambda_1$ is odd.

The same argument applied to $\lambda'$ says that every column except possibly the last has odd height and since there are multiple columns,  $h$ is also odd.

If $n$ is odd, Proposition~\ref{prop:hudak} also says that the last part of $\lambda$ is odd and the last part of $\lambda'$ is odd.  Thus every row and every column is odd.

Suppose $n$ is even. There are $h-1$ rows outside the last row.  Since $h$ is odd, $h-1$ is even, and their odd lengths have even total.  Therefore the last row has even length.  The column statement follows identically because $w$ is odd.
\end{proof}

Let $c_0$ and $c_1$ denote the numbers of nodes of residues $0$ and
$1$ in $\lambda$.  With multicharge $0$, the residue of a node $(i,j)$ is
\[
\operatorname{res}(i,j)=j-i\pmod 2.
\]
From the definitions in
\cite[(2.1), (2.4), (2.5), (2.7)]{KMR}, in our level-one $e=2$
setting we have
\[
\operatorname{cont}(\lambda)=c_0\alpha_0+c_1\alpha_1,
\]
with
\[
(\Lambda_0,\alpha_i)=\delta_{0i},
\qquad
(\alpha_0,\alpha_0)=(\alpha_1,\alpha_1)=2,
\qquad
(\alpha_0,\alpha_1)=(\alpha_1,\alpha_0)=-2.
\]
Hence
\[
(\Lambda_0,\operatorname{cont}(\lambda))=c_0
\]
and
\[
\frac12
\bigl(\operatorname{cont}(\lambda),
      \operatorname{cont}(\lambda)\bigr)
=(c_0-c_1)^2.
\]
Therefore
\begin{equation}\label{eq:defect-residue}
D(\lambda)=c_0-(c_0-c_1)^2.
\end{equation}
\begin{lemma}\label{lem:defect}
Under the assumptions above,
\[
D(\lambda)=\left\lfloor\frac n2\right\rfloor.
\]
\end{lemma}

\begin{proof}
An even length row has equally many nodes of residues zero and one.  An odd row contributes $+1$ to $c_0-c_1$ when its row number is odd and $-1$ when its row number is even.

If $n$ is odd, Lemma~\ref{lem:parity-pattern} says that all $h$ rows are odd, and $h$ is odd.  Their alternating contributions give
\[
c_0-c_1=1,
\qquad
c_0=\frac{n+1}{2}.
\]
Equation~\eqref{eq:defect-residue} yields $D(\lambda)=(n-1)/2$.

If $n$ is even, the first $h-1$ rows are odd and the last row is even.  Since $h-1$ is even, the alternating contributions cancel, so $c_0=c_1=n/2$.  Equation~\eqref{eq:defect-residue} gives $D(\lambda)=n/2$.
\end{proof}

Define
\[
E(\lambda):=n-w-h+1.
\]
This is exactly the number of nodes $(i,j)$ of the Young diagram outside the first row and the first column.  Hence
\[
E(\lambda)=0\quad\Longleftrightarrow\quad\lambda\text{ is a hook}.
\]

\begin{proposition}\label{prop:degree-formula}
If $\lambda$ is a nonhook and both Hom-spaces in \eqref{eq:two-hom-spaces} are nonzero, then
\[
a(\lambda)+a(\lambda')-D(\lambda)
=
\left\lceil\frac{E(\lambda)}2\right\rceil.
\]
Consequently
\[
\deg p_\lambda+\deg i_\lambda
=
-\left\lceil\frac{E(\lambda)}2\right\rceil<0.
\]
\end{proposition}

\begin{proof}
For any partition,
\[
a(\lambda)=\frac{n-o_r(\lambda)}2,
\qquad
a(\lambda')=\frac{n-o_c(\lambda)}2,
\]
where $o_r(\lambda)$ is the number of odd row lengths and $o_c(\lambda)$ the number of odd column heights.

If $n$ is odd, Lemma~\ref{lem:parity-pattern} gives $o_r=h$ and $o_c=w$.  Using Lemma~\ref{lem:defect},
\begin{align*}
a(\lambda)+a(\lambda')-D(\lambda)
&=n-\frac{h+w}{2}-\frac{n-1}{2}\\
&=\frac{n-h-w+1}{2}\\
&=\frac{E(\lambda)}2.
\end{align*}
Here $E(\lambda)$ is even.

If $n$ is even, $o_r=h-1$ and $o_c=w-1$, so
\begin{align*}
a(\lambda)+a(\lambda')-D(\lambda)
&=n-\frac{h+w-2}{2}-\frac n2\\
&=\frac{n-h-w+2}{2}\\
&=\frac{E(\lambda)+1}{2}.
\end{align*}
Here $E(\lambda)$ is odd.  This proves the first assertion.  The formula for $\deg p_\lambda+\deg i_\lambda$ follows from \eqref{eq:two-map-degrees}, and the final inequality holds because $E(\lambda)>0$ for a nonhook partition.
\end{proof}

\section{Obstruction to splitting}

We now record explicitly why an \emph{ungraded} splitting is
incompatible with the degree calculation above.  The point is that,
although the maps in an arbitrary ungraded splitting need not in
general be homogeneous, in our situation the relevant ordinary
Hom-spaces are one-dimensional and thus are spanned by homogeneous maps.
Consequently every nonzero map in either Hom-space is automatically
homogeneous.

\begin{lemma}\label{lem:splitting}
Let $A$ be a graded algebra, let $L$ be a nonzero graded $A$-module,
and let $M$ be a finite-dimensional graded $A$-module.  Suppose that,
after forgetting the grading,
\[
\Hom_A(L,M)=\F i,
\qquad
\Hom_A(M,L)=\F p,
\]
where
\[
i:L\longrightarrow M,
\qquad
p:M\longrightarrow L
\]
are nonzero homogeneous maps of degrees $d_i$ and $d_p$,
respectively.  If the underlying ungraded module $M$ has a direct
summand isomorphic to the underlying module $L$, then
\[
d_i+d_p=0.
\]
\end{lemma}

\begin{proof}
Suppose that the underlying ungraded module $L$ is a direct summand
of the underlying ungraded module $M$.  Then there are ordinary
$A$-module homomorphisms
\[
j:L\longrightarrow M,
\qquad
q:M\longrightarrow L
\]
such that
\[
qj=\operatorname{id}_L.
\]
In particular, both $j$ and $q$ are nonzero.  Since the two ordinary
Hom-spaces are one-dimensional, there exist nonzero scalars
$\alpha,\beta\in\F$ such that
\[
j=\alpha i,
\qquad
q=\beta p.
\]
Thus the actual splitting maps $j$ and $q$ are automatically
homogeneous, of degrees $d_i$ and $d_p$, respectively.  Moreover,
\[
\operatorname{id}_L=qj=\alpha\beta\,pi.
\]
Hence $pi$ is a nonzero scalar multiple of $\operatorname{id}_L$.  The map $pi$ is
homogeneous of degree $d_i+d_p$, whereas $\operatorname{id}_L$ is homogeneous of
degree $0$.  Since $\operatorname{id}_L\ne0$, these degrees must agree.  Therefore
\[
d_i+d_p=0,
\]
as required.
\end{proof}
\begin{proof}[Proof of Theorem~\ref{thm:main}]
Suppose, for contradiction, that the classical Specht module $\JSp^\lambda$ has a one-dimensional direct summand.  The underlying ungraded module of the graded column Specht module $\CS_\lambda$ is $\JSp^\lambda$, so $\CS_\lambda$ has, after forgetting the grading, a direct summand isomorphic to $L$.  In particular both Hom-spaces in \eqref{eq:two-hom-spaces} are nonzero.

Proposition~\ref{prop:two-degrees} shows that both Hom-spaces are one-dimensional and are spanned by the homogeneous maps $i_\lambda$ and $p_\lambda$.  Proposition~\ref{prop:degree-formula} gives
\[
\deg i_\lambda+\deg p_\lambda
=-\left\lceil\frac{E(\lambda)}2\right\rceil<0,
\]
because $\lambda$ is nonhook.  In fact, since $L$ is concentrated in degree $0$, the homogeneous endomorphism $p_\lambda i_\lambda$ of $L$ must already be zero.  More generally, Lemma~\ref{lem:splitting} says that an ungraded splitting would force the sum of the two degrees to be zero.  This contradiction proves the theorem.
\end{proof}

This means that Murphy's examples are the only Specht modules with a trivial direct summand:

\begin{corollary}\label{cor:classification}
Let $\F$ be a field of characteristic $2$, let $\lambda\vdash n$, and assume $\dim\JSp^\lambda>1$.  Then $\JSp^\lambda$ has a one-dimensional direct summand if and only if
\[
\lambda=(n-r,1^r)
\]
is a hook satisfying
\[
n\text{ is odd},\qquad r\text{ is even},\qquad
\binom{n-1}{r}\text{ is odd}.
\]
\end{corollary}

\begin{proof}
Theorem~\ref{thm:main} rules out every nonhook.  Murphy's theorem \cite[Theorem~5.5]{Murphy} thus gives a complete classification.
\end{proof}
 
Although our proof does not use the method proposed by Collins--Dodge, their criterion can be read in reverse to give information about the expansion of the distinguished fixed vector in the standard-polytabloid basis of $\JSp^\lambda$:
\begin{remark}
\label{remark:Collins-Dodge condition}
Collins and Dodge define $\lambda$ to be \emph{Lucas perfect} if both $\lambda$ and $\lambda'$ satisfy James' condition, that is, if there are nonzero maps $\F\rightarrow \JSp^\lambda$ and $\JSp^\lambda\rightarrow \F$.  Suppose that $\lambda$ is a nonhook Lucas perfect partition, and let $M^\lambda$ be the permutation module with basis the set of $\lambda$-tabloids.  Let
\[
f_\lambda:=\sum_{\{t\}}\{t\},
\]
where the sum is over all $\lambda$-tabloids.  This spans the one-dimensional fixed-point space of $M^\lambda$.  When $f_\lambda$ is expressed in the basis of standard polytabloids of $\JSp^\lambda$, the sum of its coefficients is $0$ in $\F$.  See \cite[Theorem~3.5]{CollinsDodge} for details.
\end{remark}

\begin{remark}
The proof raises two natural questions.  First, Theorem~\ref{thm:main},
together with the Collins--Dodge criterion, implies a parity statement
for the coefficients of the fixed vector in the standard-polytabloid
basis.  It would be interesting to find a direct combinatorial
explanation for this parity.

Second, the grading obstruction vanishes precisely for hooks.  Thus
the argument isolates exactly the family in which a trivial summand
can occur, but does not by itself recover Murphy's classification
within that family.  It would be interesting to determine whether the
stronger congruence conditions in Hudak's theorem can be used to give
a graded proof of Murphy's hook theorem.
\end{remark}

\bibliographystyle{alpha}
\bibliography{references}

\end{document}